\documentclass{article}
\usepackage[utf8]{inputenc}
\usepackage{amsthm}
\usepackage{amsmath}
\usepackage{amssymb}
\usepackage{amsfonts}
\usepackage{biblatex}
\usepackage{alphabeta}

\newtheorem{theorem}{Theorem}[section]
\newtheorem{corollary}{Corollary}[theorem]
\newtheorem{lemma}[theorem]{Lemma}
\newtheorem{observation}[theorem]{Observation}

\theoremstyle{definition}
\newtheorem{definition}{Definition}

\theoremstyle{remark}

\newcommand{\ch}[0]{\mathrm{ch}}

\title{Strictly \(k\)-colorable graphs}
\author{Evan Leonard}
\date{\today}

\begin{document}

    \maketitle

    \begin{abstract}
        Zhu \cite{zhu2020} introduced a refined scale of choosability in 2020 and observed that the four color theorem is tight on this scale. We formalize and explore this idea of tightness in what we call strictly colorable graphs. We then characterize  all strictly colorable complete multipartite graphs.
    \end{abstract}
    
    \section{Introduction}
    
        Vertex coloring is a widely studied area that comes in many variations.
        A proper vertex coloring of a graph \(G\) assigns colors to the vertices of \(G\) so that no two adjacent vertices have the same color.
        A graph is \(k\)-colorable if it can be properly colored with \(k\) colors.
        The chromatic number \(\chi(G)\) is the minimum \(k\) for which \(G\) is \(k\)-colorable.
        If \(\chi(G)=k\), we say \(G\) is \(k\)-chromatic.
        
        List coloring (choosability) is a popular variation of proper vertex coloring introduced in 1976 by Vizing \cite{vizing1976} and indpendently in 1979 by Erdös, Rubin, and Taylor \cite{erdos1979}.
        A \(k\)-list-assignment (\(k\)-assignment) \(L\) of \(G\) assigns sets of \(k\) colors to the vertices of \(G\).
        \(G\) is \(L\)-colorable if \(L\) exhibits a proper coloring; that is, \(G\) can be properly colored where each vertex is assigned a color from its list in \(L\).
        \(G\) is \(k\)-choosable if it is \(L\)-colorable for all \(k\)-assignments \(L\).
        The choice number \(\ch(G)\) is the minimum \(k\) for which \(G\) is \(k\)-choosable.
        
        Zhu \cite{zhu2020} introduced another variation in 2020 which refines choosability into a hierarchy of integer partitions.
        In that initial paper, he uses this new system (summarized in the next section) to extend some list coloring results as well as make connections to the List Coloring Conjecture and to signed graph coloring problems. 
        
        One observation Zhu made was based on a result of Kemnitz and Voigt \cite{voigt2018} which implies that the four color theorem is tight on his refined scale of list coloring.
        Without getting very technical yet, they found a planar graph which is only properly colorable with list assignments that are essentially equivalent to the setup of a normal proper vertex coloring.
        An example of such a list assignment would be every vertex having the same list of four colors.

        This idea of tightness was not further explored. Here we formalize it as graphs which are ``strictly colorable''. We'll explore some general observations of the idea and ultimately characterize all strictly colorable complete multipartite graphs.

    \section{Summary of Zhu's Refinement}
        
        Zhu's refinement of choosability is built using integer partitions.
        An integer partition \(λ\) of a positive integer \(k\) is a multiset of positive integers whose sum is \(k\).
        For example, \(λ=\{1,1,2,3\}\) is an integer partition of \(k=7\).
        For the integer partition \(λ = \{k_1, k_2, \dots, k_t\}\) of \(k\), a \(λ\)-assignment of a graph \(G\) is a \(k\)-assignment \(L\) of \(G\) where the colors in \(\bigcup_{v \in V(G)}L(v)\) can be partitioned into sets \(C_1, C_2, \dots, C_t\) so that for each \(v \in V(G)\) and each \(i \in \{1, 2, \dots, t\}\), \(|L(v) \cap C_i| = k_i\). \(G\) is \(λ\)-choosable if every \(λ\)-assignment of \(G\) exhibits a proper coloring.

        Let \(λ\) and \(λ'\) be integer partitions of \(k\).
        We say \(λ'\) is a refinement of \(λ\) if \(λ'\) is obtained by subdividing parts of \(λ\); e.g. \(\{1,1,3\}\) is a refinement of \(\{2,3\}\).
        It follows that if \(λ'\) is a refinement of \(λ\), then every \(λ'\)-assignment of a graph \(G\) is also a \(λ\)-assignment of \(G\).
        So, every \(λ\)-choosable graph is \(λ'\)-choosable.

        A note on notation: as mentioned earlier, we'll get to results about complete multipartite graphs.
        It's common notation in the literature to use \(K_{a*b}\) as the complete \(b\)-partite graph with parts of size \(a\).
        For example \(K_{3*5} = K_{3,3,3,3,3}\) and \(K_{5*3} = K_{5,5,5}\).
        I introduce this now because it is also very useful for integer partitions; e.g. we say \(λ = \{1*4, 2\} = \{1,1,1,1,2\}\) is an integer partition of \(6\). While there are other conventions to convey multiplicity in integer partitions, we'll stick with this one for the sake of consistency with complete multipartite graphs.

        Zhu pointed out the trivial fact that being \(\{k\}\)-choosable is equivalent to being \(k\)-choosable.
        He then proved the less obvious fact that being \(\{1*k\}\)-choosable is equivalent to being \(k\)-colorable.
        Thus \(λ\)-choosability conveniently houses \(k\)-colorability and \(k\)-choosability within the same framework.
        The integer partitions of \(k\), which are \(\{k\}\) and all refinements down to \(\{1*k\}\), reveal a complicated hierarchy of colorability.

        Refinements allow us to compare partitions of the same integer.
        Zhu introduced a partial ordering of integer partitions which allows us to compare partitions of different integers; it goes as follows.
        Let \(λ\) and \(λ'\) be integer partitions of \(k\) and \(k'\) respectively where \(k \leq k'\).
        We say \(λ \leq λ'\) if \(λ'\) is a refinement of an integer partition \(λ''\) of \(k'\) obtained from \(λ\) by increasing parts of \(λ\).
        For example, \(\{3,3\} \le \{1,1,2,4\}\); to see this, use the intermediate integer partition \(\{3,5\}\).
        Zhu then proved this important theorem.
        
        \begin{theorem}[Zhu]
            Every \(λ\)-choosable graph is \(λ'\)-choosable if and only if \(λ \leq λ'\).
        \end{theorem}
        
        Continuing our example, every \(\{3,3\}\)-choosable graph is \(\{1,1,2,4\}\)-choosable.

    \section{Strictly \(k\)-Colorable Graphs}

        Let's now give better context to Kemnitz and Voigt's result. They showed that there are planar graphs which are not \(\{1,1,2\}\)-choosable.
        That is to say, by the four color theorem there are planar graphs for which \(λ=\{1,1,1,1\}\) is the only integer partition of \(4\) for which they are \(λ\)-choosable (this is what was meant by the phrase, ``essentially equivalent to the setup of a normal proper vertex coloring,'' in the introduction).
        
        Again, Zhu points out that this makes the four color theorem tight on his refined scale of choosability.
        This idea of graphs being \(λ\)-choosable strictly for \(λ=\{1*k\}\) and no other integer partitions of \(k\) was not further explored by Zhu beyond this example.
        Here we formalize the idea.

        \begin{definition}\label{strict:def}
            A graph \(G\) is strictly \(k\)-colorable if the only integer partition \(λ\) of \(k\) for which \(G\) is \(λ\)-choosable is \(λ = \{1*k\}\).
        \end{definition}

        Here is the motivation behind the chosen terminology.
        We say ``strictly \(k\)-\emph{colorable}'' because being \(\{1*k\}\)-choosable is equivalent to being \(k\)-colorable.
        We say ``\emph{strictly} \(k\)-colorable'' because it's the only partition of \(k\) for which it is \(λ\)-choosable.
        The following observation provides a nice alternate definition of strict \(k\)-colorability.
        
        \begin{observation}
            A graph \(G\) is strictly \(k\)-colorable if and only if \(G\) is \(k\)-colorable and not \(\{1*(k-2),2\}\)-choosable.
        \end{observation}
        
        \begin{proof}
            The forward implication follows from Definition \ref{strict:def}.
            Let \(λ = \{1*(k-2),2\}\).
            Suppose \(G\) is \(k\)-colorable and not \(λ\)-choosable. Let \(λ' \ne \{1*k\}\) be an integer partition of \(k\).
            Then \(λ\) is a refinement of \(λ'\).
            Thus every \(λ\)-assignment of \(G\) is a \(λ'\)-assignment of \(G\).
            Because \(G\) is not \(λ\)-choosable, \(G\) is not \(λ'\)-choosable.
            Therefore, \(G\) is strictly \(k\)-colorable.
        \end{proof}
        
        So to show that any given graph \(G\) is strictly \(k\)-colorable, it suffices to show that \(G\) is \(k\)-colorable and then find a \(\{1*(k-2),2\}\)-assignment for which \(G\) is not properly colorable.
        Here are two more interesting and helpful observations.
        
        \begin{observation}
            If \(G\) is strictly \(k\)-colorable, then \(\chi(G) = k\).
        \end{observation}
        
        \begin{proof}
            If \(\chi(G)>k\), then \(G\) is not \(k\)-colorable.
            Let \(λ = \{1*(k-1)\}\).
            Suppose \(G\) is \(λ\)-choosable [i.e. \(\chi(G)<k\)].
            Let \(λ' = \{1*(k-2),2\}\).
            Note \(λ \le λ'\). So \(G\) is \(λ'\)-choosable, and therefore not strictly \(k\)-colorable.
        \end{proof}
        
        This means there is at most one positive integer \(k\) for which a graph \(G\) can be strictly \(k\)-colorable, that is \(k=\chi(G)\).
        If \(G\) is not strictly \(\chi(G)\)-colorable, then \(G\) is not strictly \(k\)-colorable for any \(k\in\mathbb{N}\).
        
        \begin{observation}\label{subgraph:obs}
            Suppose \(H\) is strictly \(k\)-colorable and \(H \subseteq G\). Then \(G\) is strictly \(k\)-colorable if and only if \(\chi(G) = k\).
        \end{observation}
        
        \begin{proof}
            The forward is already proven.
            Suppose \(\chi(G) = k\). Because \(H\) is not \(\{1 * (k-2), 2\}\)-choosable, neither is \(G\).
        \end{proof}
        
        With this, if you'd like to show that some graph \(G\) is strictly \(k\)-colorable, it suffices to show that a subgraph \(H \subseteq G\) with \(\chi(H)=\chi(G)\) is strictly \(k\)-colorable.
        Put another way, if you show that some graph \(H\) is strictly \(k\)-colorable, then you get every \(k\)-chromatic graph containing \(H\) for free.
        
        Some more can be said in general about the lowest values of \(k\).
        If \(G\) is strictly \(1\)-colorable, then \(G\) is an independent set.
        Because the only integer partition of \(1\) is \(\{1\}\), all independent sets are strictly \(1\)-colorable.
        If \(G\) is strictly \(2\)-colorable, then \(G\) is bipartite.
        The only two integer partitions of \(2\) are \(\{1,1\}\) and \(\{2\}\).
        Thus, a bipartite graph is strictly \(2\)-colorable if and only if it is not \(2\)-choosable.
        Erdös, Rubin, and Taylor \cite{erdos1979} characterized all \(2\)-choosable graphs, so all strictly \(2\)-colorable graphs are characterized.
        As with many problems, the fun starts with \(k\ge3\), so from here on, we will consider strict \(k\)-colorability only for \(k \ge 3\).

        \section{Strictly \(k\)-Colorable Complete \(k\)-Partite Graphs}

    Complete multipartite graphs are typically of interest when studying choosability.
    Their nice structure can lend to tidy results.
    Erdös, Rubin, and Taylor \cite{erdos1979} proved that \(\ch(K_{2*k})=k\).
    \(K_{2*k}\) is called a chromatic-choosable graph since \(\chi(K_{2*k}) = \ch(K_{2*k})\).
    This, of course, disqualifies it from being strictly \(k\)-colorable.
    In this context, strictly \(k\)-colorable graphs can be thought of as one end of the refinement spectrum and chromatic-choosable graphs as the opposite end.
    
    Kierstead proved in 2000 \cite{kierstead2000} that \(\ch(K_{3*k}) = \lceil(4k - 1)/3\rceil\) and proved with Salmon and Wang in 2016 \cite{kierstead2016} that \(\ch(K_{4*k}) = \lceil(3k-1)/2\rceil\).
    These graphs with tidy choice numbers are not chromatic-choosable.
    Might they be strictly \(k\)-colorable?
    Yes.
    
    \begin{lemma}\label{k3k:lem}
        Let \(k\ge3\). \(K_{3*k}\) is strictly \(k\)-colorable.
    \end{lemma}
    
    \begin{proof}
        \(K_{3*k}\) is certainly \(k\)-colorable.
        Let \(λ_k = \{1*(k-2),2\}\).
        It suffices to show that \(K_{3*k}\) is not \(λ_k\)-choosable.
        Let \(V_1, V_2, \dots, V_k\) be the partite sets of \(K_{3*k}\).
        Define \(L_k\) to be the following \(k\)-assignment:
        \begin{align*}
            &L_k(V_1)       &&L_k(V_2)          &&          &&L_k(V_k)\\
            &\{0,1\} \cup A &&\{0,1\} \cup A    &&          &&\{0,1\} \cup A\\
            &\{0,2\} \cup A &&\{0,2\} \cup A    &&\cdots    &&\{0,2\} \cup A\\
            &\{1,2\} \cup A &&\{1,2\} \cup A    &&          &&\{1,2\} \cup A
        \end{align*}
        Where \(A = \{3, \dots, k\}\).
        Note the dots (\dots) mean count up by 1 between \(3\) and \(k\).
        For example, if \(k=7\), the second vertex in \(V_1\) has the list assignment \(\{0,2,3,4,5,6,7\}\).
        Let \(C_1 = \{0,1,2\}\) and \(C_i = \{i+1\}\) for \(2 \le i \le k-1\).
        Then \(|L_k(v) \cap C_1| = 2\) and \(|L_k(v) \cap C_i|= 1\) for all \(v \in V(K_{3*k})\) and \(2 \le i \le k-1\).
        Thus, \(L_k\) is a \(λ_k\)-assignment to \(K_{3*k}\).
        
        There are \(k\) partite sets and \(k-1\) color groups.
        For \(2 \le i \le k-1\), the colors of \(C_i\) can appear on at most \(1\) partite set of \(K_{3*k}\).
        The colors of \(C_1\) cannot fully color \(2\) partite sets.
        Hence, all the color groups together can fully color at most \(k-1\) partite sets simultaneously.
        So, \(K_{3*k}\) is not \(L_k\)-colorable, and therefore not \(λ_k\)-choosable.
        Therefore, \(K_{3*k}\) is strictly \(k\)-colorable.
    \end{proof}

    By Observation \ref{subgraph:obs}, so is \(K_{4*k}\).
    It's worth noting that this lemma is also true for \(k = 1,2\), but since independent sets and bipartite graphs are solved, we only care for \(k \ge 3\).

    It turns out that we can characterize all strictly \(k\)-colorable complete \(k\)-partite graphs.
    The strategy is to find a set of subgraphs which are strictly \(k\)-colorable.
    Then we will show it is necessary and sufficient for any strictly \(k\)-colorable complete \(k\)-partite graph to contain one of these subgraphs.
    There are three such subgraphs in total.
    Our first is the previously mentioned \(K_{3*k}\).
    
    To introduce the remaining two, we'll make use of a result by Hoffman and Johnson \cite{johnson1993}.
    They showed that there is a unique uncolorable \(m\)-assignment (up to relabeling) of \(K_{m,n}\) when \(n=m^m\).
    For \(K_{2,4}\) with partite sets \(V_1\) and \(V_2\), that unique assignment is \(L(V_1) = \{\{1,2\},\{3,4\}\}\), \(L(V_2) = \{\{1,3\},\{1,4\},\{2,3\},\{2,4\}\}\).
    We'll call this the ``unique bad \(2\)-assignment of \(K_{2,4}\).''
    
    \begin{lemma}\label{k246:lem}
        Let \(k\ge3\). \(K_{2, 4, 6*(k-2)}\) is strictly \(k\)-colorable.
    \end{lemma}
    
    \begin{proof}
        Let \(G_k = K_{2, 4, 6*(k-2)}\).
        \(G_k\) is certainly \(k\)-colorable.
        Let \(λ_k = \{1*(k-2),2\}\).
        It suffices to show that \(G_k\) is not \(λ_k\)-choosable.
        Let \(V_1, V_2, \dots, V_k\) be the partite sets of \(G_k\) such that \(|V_1|=2\), \(|V_2|=4\), and \(|V_i|=6\) for \(3 \le i \le k\).
        Define \(L_k\) to be the following \(k\)-assignment on \(G_k\):
        \begin{align*}
            &L_k(V_1)       &&L_k(V_2)          &&L_k(V_3)          &&          &&L_k(V_k)\\
            &\{1,2\} \cup A &&\{1,3\} \cup A    &&\{1,3\} \cup A    &&          &&\{1,3\} \cup A\\
            &\{3,4\} \cup A &&\{1,4\} \cup A    &&\{1,4\} \cup A    &&          &&\{1,4\} \cup A\\
            &               &&\{2,3\} \cup A    &&\{2,3\} \cup A    &&\cdots    &&\{2,3\} \cup A\\
            &               &&\{2,4\} \cup A    &&\{2,4\} \cup A    &&          &&\{2,4\} \cup A\\
            &               &&                  &&\{1,2\} \cup A    &&          &&\{1,2\} \cup A\\
            &               &&                  &&\{3,4\} \cup A    &&          &&\{3,4\} \cup A
        \end{align*}
        Where \(A = \{5, \dots, k+2\}\).
        Let \(C_1 = \{1,2,3,4\}\) and let \(C_i = \{i+3\}\) for \(2 \le i \le k-1\).
        Then \(|L_k(v) \cap C_1| = 2\) and \(|L_k(v) \cap C_i| = 1\) for all \(v \in V(G_k)\) and \(2 \le i \le k-1\).
        Thus, \(L_k\) is a \(λ_k\)-assignment to \(G_k\).
        
        Notice there are \(k\) partite sets and \(k-1\) color groups.
        In a proper \(L_k\)-coloring of \(G_k\), each color group \(C_i\) where \(2 \le i \le k-1\) can be seen on at most one partite set.
        This means in a proper \(L_k\)-coloring of \(G_k\), at least two partite sets must only see colors from \(C_1\).
        But between every pair of partite sets, their colors from \(C_1\) contains the unique bad 2-assignment of \(K_{2,4}\).
        Thus you can't completely color any pair of partite sets using only \(C_1\), so \(G_k\) is not \(L_k\)-colorable.
        Hence, \(G_k\) is not \(λ_k\)-choosable.
        Therefore, \(G_k\) is strictly \(k\)-colorable.
    \end{proof}
    
    \begin{lemma}\label{k255:lem}
        Let \(k\ge3\). \(K_{2, 5*(k-1)}\) is strictly \(k\)-colorable.
    \end{lemma}
    
    \begin{proof}
        Let \(G_k = K_{2, 5*(k-1)}\).
        \(G_k\) is certainly \(k\)-colorable.
        Let \(λ_k = \{1*(k-2),2\}\).
        It suffices to show that \(G_k\) is not \(λ_k\)-choosable. Let \(V_1, V_2, \dots, V_k\) be the partite sets of \(G_k\) such that \(|V_1|=2\), and \(|V_i|=5\) for \(2 \le i \le k\).
        Define \(L_k\) to be the following \(k\)-assignment on \(G_k\):
        \begin{align*}
            &L_k(V_1)       &&L_k(V_2)          &&          &&L_k(V_k)\\
            &\{1,2\} \cup A &&\{1,3\} \cup A    &&          &&\{1,3\} \cup A\\
            &\{3,4\} \cup A &&\{1,4\} \cup A    &&          &&\{1,4\} \cup A\\
            &               &&\{2,3\} \cup A    &&\cdots    &&\{2,3\} \cup A\\
            &               &&\{2,4\} \cup A    &&          &&\{2,4\} \cup A\\
            &               &&\{1,2\} \cup A    &&          &&\{1,2\} \cup A
        \end{align*}
        Where \(A = \{5, \dots, k+2\}\).
        Let \(C_1 = \{1,2,3,4\}\) and let \(C_i = \{i+3\}\) for \(2 \le i \le k-1\).
        Then \(|L_k(v) \cap C_1| = 2\) and \(|L_k(v) \cap C_i| = 1\) for all \(v \in V(G_k)\) and \(2 \le i \le k-1\).
        Thus, \(L_k\) is a \(λ_k\)-assignment to \(G_k\).
        
        Just as in the previous proofs, there are \(k\) partite sets and \(k-1\) color groups.
        The color groups of size 1 can together color at most \(k-2\) partite sets, leaving at least \(2\) partite sets left to be colored by \(C_1\).
        However, notice between any pair of partite sets, it is impossible for \(C_1\) to be used alone.
        Therefore, \(G_k\) is strictly \(k\)-colorable.
    \end{proof}
    
    We'll use these three strictly \(k\)-colorable complete \(k\)-partite graphs to characterize all such graphs.
    Before getting to it, we'll make use of another way that colorability can be thought of in terms of integer partitions and how that relates to our current notion of \(λ\)-choosability.
    
    \begin{definition}
        Let \(λ = \{k_1, k_2, \dots, k_t\}\) be an integer partition of \(k\).
        A graph \(G\) is \(λ\)-partitionable if there exists a partition \(V_1, V_2, \dots, V_t\) of the vertex set \(V(G)\) such that \(G[V_i]\) is \(k_i\)-choosable for \(1 \le i \le t\).
        Such a partition of \(V(G)\) is called a \(λ\)-partition.
    \end{definition}

    This idea of distinguishing \(λ\)-partitionability from \(λ\)-choosability was introduced to me by Greg Puleo in unpublished work he did along with Dan Cranston. Here is an observation of how the two ideas compare.
    
    \begin{observation}
        If \(G\) is \(λ\)-partitionable, then \(G\) is \(λ\)-choosable.
    \end{observation}
    
    \begin{proof}
        Let \(λ = \{k_1, \dots, k_t\}\) and let \(V_1, \dots, V_t\) be a \(λ\)-partition of \(V(G)\).
        Let \(L\) be a \(λ\)-assignment of \(G\) with color groups \(C_1, \dots, C_t\).
        Because \(G[V_i]\) is \(k_i\)-choosable, the vertices of \(V_i\) can be properly colored with the colors assigned to it from \(C_i\).
        Because all \(C_i\) are disjoint, \(G\) is \(L\)-colorable and therefore \(λ\)-choosable since \(L\) is an arbitrary \(λ\)-assignment.
    \end{proof}
    
    \begin{corollary}\label{tkc-imp-part:cor}
        If \(G\) is \(\{1*(k-2),2\}\)-partitionable, then \(G\) is not strictly \(k\)-colorable.
    \end{corollary}
    
    It's nice to note that this characterizes complete graphs for free.
    
    \begin{corollary}
        For \(n\ge2\), \(K_n\) is not strictly \(n\)-colorable.
    \end{corollary}
    
    This idea of \(λ\)-partitionability is a quick way to show certain graphs are not strictly \(k\)-colorable.
    This will be utilized for our final theorem.
    
    \begin{theorem}
        Let \(k\ge3\) and \(G_k\) be a complete \(k\)-partite graph. \(G_k\) is strictly \(k\)-colorable if and only if \(G_k\) contains at least one of \(K_{3*k}\), \(K_{2,4,6*(k-2)}\), or \(K_{2, 5*(k-1)}\) as a subgraph.
    \end{theorem}
    
    \begin{proof}
        The backwards implication follows from Observation \ref{subgraph:obs} and Lemmas \ref{k3k:lem}, \ref{k246:lem}, and \ref{k255:lem}.
        
        Let \(G_k = K_{a_1, a_2, \dots, a_k}\) such that \(a_1 \le a_2 \le \cdots \le a_k\).
        Let \(V_1, \dots, V_k\) be the partite sets of \(G_k\) such that \(|V_i| = a_i\) for \(1 \le i \le k\).
        Let \(λ_k = \{1*(k-2),2\}\).
        Suppose \(G_k\) contains none of \(K_{3*k}\), \(K_{2,4,6*(k-2)}\), and \(K_{2, 5*(k-1)}\) as a subgraph.
        Then we have the following two cases:
        
        \textbf{Case 1:} \(a_1 = 1\), or \(a_1 = 2\) and \(a_2 \le 3\).
        In this case, \(G_k[V_1 \cup V_2]\) is \(2\)-choosable.
        So \(G_k\) is \(λ_k\)-partitionable, and hence not strictly \(k\)-colorable.
        
        \textbf{Case 2:} \(a_1 = 2\), \(a_2 = 4\), and \(a_3 \le 5\).
        It suffices to only consider \(a_3 = 5\).
        Let \(L_k\) be a \(λ_k\)-assignment of \(G_k\) with color groups \(C_1, C_2, \dots, C_{k-1}\) such that \(|L_k(v) \cap C_1| = 2\) and \(|L_k(v) \cap C_i| = 1\) for all \(v \in V(G_k)\) and for \(2 \le i \le (k-1)\).
        Color \(V_i\) with its colors from \(C_{i-1}\) for \(3 \le i \le k\).
        If \(V_1\) and \(V_2\) can be properly colored with \(C_1\), then we're done.
        If not, \(C_1\) on \(V_1\) and \(V_2\) is the unique bad \(2\)-assignment on \(K_{2,4}\).
        \begin{align*}
            &L_k(V_1) \cap C_1  &&L_k(V_2) \cap C_1\\
            &\{1,2\}            &&\{1,3\}\\
            &\{3,4\}            &&\{1,4\}\\
            &                   &&\{2,3\}\\
            &                   &&\{2,4\}
        \end{align*}
        Uncolor \(V_3\) with \(C_2\) and color \(V_2\) with \(C_2\).
        If \(V_1\) and \(V_3\) can't be colored with \(C_1\), then \(C_1\) on \(V_1\) and \(V_3\) contains the unique bad \(2\)-assignment on \(K_{2,4}\).
        \begin{align*}
            &L_k(V_1) \cap C_1  &&L_k(V_2) \cap C_1 &&L_k(V_3) \cap C_1\\
            &\{1,2\}            &&\{1,3\}           &&\{1,3\}\\
            &\{3,4\}            &&\{1,4\}           &&\{1,4\}\\
            &                   &&\{2,3\}           &&\{2,3\}\\
            &                   &&\{2,4\}           &&\{2,4\}\\
            &                   &&                  &&\{a,b\}
        \end{align*}
        Note that \(a\) and \(b\) are unknown colors from \(C_1\).
        Uncolor \(V_2\) with \(C_2\) and color \(V_1\) with \(C_2\).
        If \(1 \in \{a,b\}\), then we can color \(V_2\) with \(3,4\) and \(V_3\) with \(1,2\).
        We can do the same if \(2 \in \{a,b\}\).
        Otherwise, we can color \(V_2\) with \(1,2\) and \(V_3\) with \(3,4,a\) since \(a \ne 1,2\).
        Hence, \(G_k\) is \(λ_k\)-choosable, and therefore not strictly \(k\)-colorable.
    \end{proof}

\section*{Acknowledgements}

    I thank my former advisor Greg Puleo for introducing me to Zhu's paper on \(λ\)-choosability.
    I thank my current advisor Pete Johnson for his guidance and helpful discussions.

    \printbibliography
        
\end{document}